\documentclass[12pt]{amsart}
\usepackage{amssymb,latexsym}
\usepackage{enumerate}

\makeatletter
\@namedef{subjclassname@2010}{%
  \textup{2010} Mathematics Subject Classification}
\makeatother

\newtheorem{thm}{Theorem}[section]
\newtheorem{cor}[thm]{Corollary}
\newtheorem{lemma}[thm]{Lemma}
\newtheorem{prop}[thm]{Proposition}

\theoremstyle{definition}

\newtheorem{rem}[thm]{Remark}

\numberwithin{equation}{section}

\newcommand{\Ba}[1]{\begin{array}{#1}}
\newcommand{\Ea}{\end{array}}
\newcommand{\Be}{\begin{equation}}
\newcommand{\Ee}{\end{equation}}
\newcommand{\Bea}{\begin{eqnarray}}
\newcommand{\Eea}{\end{eqnarray}}
\newcommand{\Beas}{\begin{eqnarray*}}
\newcommand{\Eeas}{\end{eqnarray*}}
\newcommand{\Benu}{\begin{enumerate}}
\newcommand{\Eenu}{\end{enumerate}}
\newcommand{\Bi}{\begin{itemize}}
\newcommand{\Ei}{\end{itemize}}

\begin{document}

\titlepage
\title[About a conjecture of Lieb-Solovej]{About a conjecture of Lieb-Solovej}

\author[D. B\'ekoll\`e]{David B\'ekoll\`e}
\address{Department of Mathematics, Faculty of Science, University of Yaound\'e I, P.O. Box 812, Yaound\'e, Cameroon }
\email{{\tt dbekolle@gmail.com}}
\author[J. Gonessa]{Jocelyn Gonessa}
\address{Universit\'e de Bangui, Facult\'e des Sciences, D\'epartement de Math\'ematiques et Informatique, BP. 908, Bangui, R\'epublique Centrafricaine}
\email{\tt gonessa.jocelyn@gmail.com}
\author[B.F. Sehba]{Beno\^it F. Sehba}
\address{Department of Mathematics, University of Ghana, Legon LG 62, Accra, Ghana}
\email{{\tt bfsehba@ug.edu.gh}}


\begin{abstract}
Very recently, E. H. Lieb and J. P. Solovej stated a conjecture about the constant of embedding between two Bergman spaces of the upper-half plane. A question in relation with a Werhl-type entropy inequality for the affine $AX+B$ group. More precisely, that for any holomorphic function $F$ on the upper-half plane $\Pi^+$, $$\int_{\Pi^+}|F(x+iy)|^{2s}y^{2s-2}dxdy\le \frac{\pi^{1-s}}{(2s-1)2^{2s-2}}\left(\int_{\Pi^+}|F(x+iy)|^2 dxdy\right)^s
$$
for $s\ge 1$, and the constant $\frac{\pi^{1-s}}{(2s-1)2^{2s-2}}$ is sharp. We prove differently that the above holds whenever $s$ is an integer and we prove that it holds when $s\rightarrow\infty$. We also prove that when restricted  to powers of the Bergman kernel, the conjecture holds. We next study the case where $s$ is close to $1.$ Hereafter, we transfer the conjecture to the unit disc where we show that the conjecture holds when restricted to analytic monomials. Finally, we overview the bounds we obtain in our attempts to prove the conjecture.
\end{abstract}
\maketitle

\section{Introduction}
Let us denote the upper-half plane by $\Pi^+:=\{x+iy\in \mathbb{C}: y>0\}$. Let $\nu>-1$ and $1\le p<\infty$. We denote by $A^p_\nu (\Pi^+)$ the weighted Bergman space consisting of holomorphic functions $F$ on $\Pi^+$ such that
$$||F||_{A^p_\nu }:=\left (\int_{\Pi^+} |F(x+iy)|^p \hskip 1truemm y^{\nu}dx dy\right )^{\frac 1p} <\infty.$$
In \cite{LS}, in relation with a Werhl-type inequality for the affine $AX+B$ group, E. H. Lieb and J. P. Solovej formulated the following conjecture.
\vskip .1cm
\noindent
{\bf Conjecture:} Let $s\ge 1$. Then for any $F\in A^2(\Pi^+)=A^2_0(\Pi^+)$, the following inequality holds
\Be\label{eq:bsconject}
\int_{\Pi^+}|F(x+iy)|^{2s}y^{2s-2}dxdy\le \frac{\pi^{1-s}}{(2s-1)2^{2s-2}}\left(\int_{\Pi^+}|F(x+iy)|^2 dxdy\right)^s.
\Ee
with equality if $F$ is proportional to $(z-z_0)^{-2}$. 
\vskip .1cm
The above inequality just translates the embedding of $A^2(\Pi^+)$ into $A^{2s}_{2s-2}(\Pi^+)$ which is pretty easy to establish if one does not pay attention to the constant. Theorem 3 in \cite{LS} says that the inequality (\ref{eq:bsconject}) holds when $s$ is an integer. The conjecture is then for the non-integer values of $s$. 
\vskip .1cm
We do not know yet how to prove this conjecture in general but provide a different proof in the case where $s$ is an integer and prove that  in the general case, the constant in (\ref{eq:bsconject}) is sharp. More precisely, we prove the following.
\begin{thm}\label{thm:integercase}
The inequality (\ref{eq:bsconject}) holds for all positive integers $s$. 
\end{thm}
In fact, Bayart, Brevig, Haimi, Ortega-Cerd\`a and Perfekt \cite{BBHOP} also proved the same theorem and even settled the conjecture for $s=n+\frac 12 \hskip 2truemm (n=1, 2, 3,\cdots).$ Their setting is the unit disc and we shall state their results precisely in Section 4, where we transfer the conjecture from the upper half-plane to the unit disc. Analogous conjectures on contractive inequalities for Hardy spaces were studied in \cite{BOSZ}. Earlier sharp inequalities were obtained in \cite{B}.
\vskip .1cm
In the sequel, for short, we adopt the following notation
$$C_s=\frac {\pi^{1-s}}{(2s-1)2^{2s-2}}.$$

We next prove that the above result holds whenever $s\rightarrow\infty$. 
\begin{prop}
The conjecture of Lieb-Solovej is asymptotically true, in the sense that
\begin{equation}
\lim \limits_{s\rightarrow \infty}
\frac {\max \limits_{F\in A^2 (\Pi^+), F\not \equiv 0} \frac {\int_{\Pi^+} |F(x+iy)|^{2s}y^{2s-2}dxdy}{\left (\int_{\Pi^+} |F(x+iy)|^2dxdy \right )^s}}{\frac {\pi^{1-s}}{(2s-1)2^{2s-2}}}=1.
\end{equation}
\end{prop}
We also obtain that when restricted only to powers of the Bergman kernel, the conjecture holds.
\begin{thm}
\begin{equation}\label{power}
\int_{\Pi^+} \frac {y^{2s-2}}{\left \vert x+i(y+1)\right \vert^{2rs}}dxdy\leq C_s \left (\int_{\Pi^+} \frac {dxdy}{\left \vert x+i(y+1)\right \vert^{2r}}   \right )^s.
\end{equation}
Moreover, equality holds in (\ref{power}) if and only if $r=2.$
\end{thm}
A direct consequence of Theorem 1.3 is the following.
\begin{cor}
May the Lieb-Solovej conjecture be true, the Bergman kernel functions $\frac 1\pi (z-\bar z_0)^{2}\hskip 2truemm (z_0\in \Pi^+)$ will be  maximizing functions of this extremum problem.
\end{cor}
We finally provide an equivalent form for the Lieb-Solovej conjecture for $s$ close to $1.$
\begin{thm}
Let $F$ be a holomorphic function in $\Pi^+$ such that
\begin{equation*}
\int_{\Pi^+} |F(x+iy)|^2dxdy =1.
\end{equation*}
The following two assertions are equivalent.
\begin{enumerate}
\item
the Lieb-Solovej conjecture is valid for $s$ close to $1,$ i.e. there exists $s_0>1$ such that for every $s\in (1, s_0),$  we have
\begin{equation}\label{local}
\int_{\Pi^+}|F(x+iy)|^{2s}y^{2s-2}dxdy\le \frac{\pi^{1-s}}{(2s-1)2^{2s-2}}.
\end{equation}
\item
\begin{equation}\label{step}
\int_{\Pi^+} \log \left [\frac 1{2\sqrt \pi|F(x+iy)|y}  \right ]|F(x+iy)|^{2}dxdy\geq 1
\end{equation}
\end{enumerate}
\end{thm}
Unfortunately, we only succeeded in proving the following partial result.
 \begin{thm}
Let $F\in A^2 (\Pi^+)$ such that $\left \Vert F\right \Vert_{A^2 (\Pi^+)}=1.$ We have
$$\int_{\Pi^+} \log \left [\frac 1{2\sqrt \pi|F(x+iy)|y}  \right ]|F(x+iy)|^{2}dxdy\geq \frac {\log 3}2\approx 0.5493.$$
\end{thm}
The plan of the paper is as follows. In Section 2, we prove estimates for the standard weighted Bergman kernels from which we deduce a related estimate (Corollary 2.3). There, we prove Theorem 1.1 and Proposition 1.2. In Section 3, we prove Theorem 1.3, Theorem 1.4 and Theorem 1.5. In Section 4, we transfer the conjecture and our results to the unit disc and we show that the conjecture holds when restricted to analytic monomials. Finally, in Section 5, we overview the bounds we obtain in our attempts to prove the Lieb-Solovej conjecture. 

\section{The cases where $s$ is an integer and $s\rightarrow \infty$.}
\subsection{Estimates for the standard weighted Bergman kernels}
For $\nu>0,$ we recall the expression of the standard weighted Bergman kernel $B_\nu (z, w)$ of $\Pi^+$ (cf. e.g. \cite{BBGNPR}):
\begin{equation}\label{berg}
B_\nu (z, w)=\frac {2^{\nu-1}\nu}\pi \left (\frac {z-\bar w}i \right )^{-(\nu+1)} \quad \quad (z, w\in \Pi^+).
\end{equation} 
In particular, this kernel has the following reproducing property:
\begin{equation}\label{repro}
F(z)=\int_{\Pi^+} B_\nu (z, u+iv)F(u+iv)v^{\nu-1}dudv \quad \quad (F\in A^2 (\Pi^+)).
\end{equation}

We shall need the following result from \cite{BS}. We give its proof for completeness.
\begin{prop}\label{prop}
Let $r>0, \hskip 1truemm t> -1$ with $2r-t> 2.$ Then for $x+iy\in \Pi^+,$
\begin{equation}\label{est}
\int_{\Pi^+} \frac {v^tdudv}{|x+iy-u+iv|^{2r}}=\frac {4\pi \Gamma (1+t)\Gamma (2r-t-2)}{2^{2r}(\Gamma (r))^2 y^{2r-t-2}}.
\end{equation}
\end{prop}

\begin{proof}
We first integrate with respect to $u.$ We have
$$I_{y, v}:=\int_{-\infty}^\infty \frac {du}{|x+iy-u+iv|^{2r}}=\int_{-\infty}^\infty \frac {du}{(x-u)^2+(y+v)^2)^r}.$$
We apply the change of variable $u\mapsto s=\frac {x-u}{y+v}, \hskip 1truemm du=-(y+v)ds;$ we get:
$$I_{y, v}=\frac {c_r}{(y+v)^{2r-1}},$$
where 
$$c_r:=\int_{-\infty}^\infty \frac 1{(s^2+1)^r}ds.$$ 
We next integrate with respect to $v;$ applying the change of variable $v\mapsto \tau=\frac vy, dv =yd\tau,$ we obtain:
$$\int_0^\infty \frac {v^tdv}{(y+v)^{2r-1}}=d_{r, t}\frac 1{y^{2r-t-2}},$$
where
$$d_{r, t}:=\int_0^\infty \frac {\tau^td\tau}{(1+\tau)^{2r-1}}.$$
We have shown that 
$$\int_{\Pi^+} \frac {v^tdudv}{|x+iy-u+iv|^{2r}}=c_rd_{r, t}\frac 1{y^{2r-t-2}}.$$
To conclude the proof, we need the following lemma.

\begin{lemma}
We have the following equalities.
$$c_r =\frac {\sqrt \pi \Gamma (r-\frac 12)}{\Gamma (r)}$$
and 
$$d_{r, t}=\frac {\Gamma (t+1)\Gamma (2r-t-2)}{\Gamma (2r-1)}.$$
\end{lemma}

\begin{proof}[Proof of the lemma]
We have:
$$c_r=2\int_0^\infty \frac 1{(s^2+1)^r}ds.$$
Applying the change of variable $s=\sqrt \sigma,\hskip 1truemm ds=\frac {d\sigma}{2\sqrt \sigma},$ we get:
$$c_r=\int_0^\infty \frac {\sigma^{-\frac 12}}{(\sigma+1)^r}d\sigma.$$
We next apply the following well known formula:
$$B(x, y)=\int_0^\infty \frac {t^{x-1}}{(t+1)^{x+y}}dt \quad \quad (x>0, \hskip 1truemm y>0),$$
where $B(\cdot, \cdot)$ denotes the beta function. We obtain:
$$c_r=B\left (\frac 12, r-\frac 12\right )=\frac {\Gamma (\frac 12)\Gamma (r-\frac 12)}{\Gamma (r)}=\frac {\sqrt \pi \Gamma (r-\frac 12)}{\Gamma (r)}$$
and
$$d_{r, t}=B(t+1, 2r-t-2)=\frac {\Gamma (t+1)\Gamma (2r-t-2)}{\Gamma (2r-1)}.$$
\end{proof}
It follows from this lemma that
$$c_rd_{r, t}=\frac {\sqrt \pi \Gamma (r-\frac 12)}{\Gamma (r)}\frac {\Gamma (t+1)\Gamma (2r-t-2)}{\Gamma (2r-1)}.$$
We finally apply the duplication formula (cf. e.g. \cite{A}):
$$\Gamma (x)\Gamma (x+\frac 12)=\frac {\sqrt \pi\Gamma (2x)}{2^{2x-1}}\quad \quad (x>0).$$
For $x=r-\frac 12,$ we get:
$$\frac {\Gamma (r-\frac 12)}{\Gamma (2r-1)}=\frac {\sqrt \pi}{2^{2r-2}\Gamma (r)}.$$
We then conclude that
$$c_rd_{r, t}=\frac {4\pi \Gamma (1+t)\Gamma (2r-t-2)}{2^{2r}(\Gamma (r))^2}.$$
The result follows. 
\end{proof}

From the reproducing property (\ref{repro}) of the Bergman kernel and the previous proposition, we deduce the following estimate.

\begin{cor}\label{cor:pointwiseimproved}
For every $F\in A^2 (\Pi^+),$ we have
$$\sup \limits_{x+iy\in \Pi^+} \left \vert  F(x+iy)\right \vert y\leq \frac 1{2\sqrt \pi}\left \Vert F \right \Vert_{ A^2 (\Pi^+)}.$$
This estimate is sharp as equality holds for $F_0 (z)=\frac 1 {(z+i)^{2}}.$
\end{cor}

\begin{proof}
From (\ref{berg}) and (\ref{repro}) in the particular case $\nu=1$, it follows that
$$|F(x+iy)|\leq\frac 1\pi\int_{\Pi^+} \frac {1}{|x+iy-u+iv|^{2}}|F(u+iv)|dudv.$$
Applying the Schwarz inequality implies
$$|F(x+iy)|\leq\frac 1\pi\left (\int_{\Pi^+} \frac 1{|x+iy-u+iv|^{4}}dudv\right )^{\frac 12}\left \Vert F \right \Vert_{ A^2 (\Pi^+)}.$$
Applying the previous proposition with $t=0$ and $r=2,$ we obtain:
\begin{equation}\label{estberg}
\int_{\Pi^+} \frac 1{|x+iy-u+iv|^4}dudv=\frac \pi {4y^2}.
\end{equation}
So
$$\sup \limits_{x+iy\in \Pi^+} \left \vert F(x+iy)\right \vert y\leq \frac 1{2\sqrt \pi}\left \Vert F \right \Vert_{ A^2 (\Pi^+)}.$$
For $F_0 (x+iy)=\frac 1 {(z+i)^{2}},$ we have
$$\sup \limits_{x+iy\in \Pi^+} \left \vert  F_0(x+iy)\right \vert y\geq \left \vert F(i)\right \vert\times 1=\frac 14=\frac 1{2\sqrt \pi}\left \Vert F_0 \right \Vert_{ A^2 (\Pi^+)},$$
where the latter equality follows from (\ref{estberg}). This proves that the estimate is sharp.
\end{proof}
\vskip 1truemm

\subsection{Proof of Theorem 1.1}
\begin{rem}
{\rm The conjecture of Lieb-Solovej \cite{LS} can be written in the form of the following extremum problem. For all $s> 1,$
\begin{equation}\label{bsconject}
\max \limits_{F\in A^2 (\Pi^+), \hskip 1truemm F\not \equiv 0} \frac {\int_{\Pi^+} |F(x+iy)|^{2s}y^{2s-2}dxdy}{\left (\int_{\Pi^+} |F(x+iy)|^2dxdy \right )^s}=\frac {\pi^{1-s}}{(2s-1)2^{2s-2}}.
\end{equation}
May this conjecture be true, Theorem 1.3 says that  the Bergman kernel functions $\frac 1\pi (z-\bar z_0)^{2}\hskip 2truemm (z_0\in \Pi^+)$ will be  maximizing functions of this problem.} 
\end{rem}
\noindent
\textbf{Suggestion.} Prove that the maximum is attained (easy). Show next that the maximum functions are 'unique' up to a multiplicative constant and are equal to Bergman kernel functions.

Lieb and Solovej have proved the conjecture when $s$ is an integer greater than $1.$ To this aim, they used representation theory. We provide a direct proof below based on Fourier-Laplace representation of functions in $A^2 (\Pi^+)$.

\begin{proof}[Proof of Theorem 1.1]
We start this proof by recalling the following Paley-Wiener theorem for Bergman spaces (see \cite[Theorem 1]{D2}).
\begin{prop} \label{PW} 
Let $F$ be a holomorphic function on $\Pi^+$, and let $\alpha>-1$. Then the following assertions are equivalent.
\begin{enumerate}
\item[\rm (1)]
$F$ belongs to the weighted Bergman space $A^2_\alpha (\Pi^+).$ 
\item [\rm (2)]
There exists a function $f:(0, \infty)\rightarrow \mathbb C$ satisfying the estimate $\int_0^\infty \frac{|f(t)|^2}{t^{\alpha+1}}dt <\infty,$ such that
$$F(z) = \int_0^\infty f(t)e^{itz}dt \quad \quad \quad (z\in \Pi^+).$$
\end{enumerate}
In this case, 
$$||F||_{A^2_\alpha (\Pi^+)}^2 =  \frac{2\pi \Gamma (\alpha +1)}{2^{\alpha+1}}\int_0^\infty |f(t)|^2 \frac {dt}{t^{\alpha+1}},$$ 
where $\Gamma$ is the usual gamma function.
\end{prop}

The following lemma is the key of our proof.
\begin{lemma}\label{keylemma}
For any $u>0$, if $I_n$ is the integral defined by
$$I_n(u):=\int_{A_n}\left(u-\sum_{j=1}^nt_j\right)\left(\prod_{j=1}^nt_j\right)dt_1\ldots t_n$$
where $A_n:=\{(t_1,\ldots,t_n)\in (0,\infty)^n:\,u-\sum_{j=1}^nt_j>0\}$, then $I_n$ converges and $$I_n(u)=\frac{u^{2n+1}}{\Gamma(2n+2)}.$$
\end{lemma}
\begin{proof}
For $k=0,\ldots,n-1,$ define $I_{n-k}$ by 
$$I_{n-k}(u):=\int_{A_{n,k}}\left(u-\sum_{j=1}^{n-k}t_j\right)^{2k+1}\left(\prod_{j=1}^{n-k}t_j\right)dt_1\ldots t_{n-k}$$
where $A_{n,k}:=\{(t_1,\ldots,t_{n-k})\in (0,\infty)^n:\,u-\sum_{j=1}^{n-k}t_j>0\}$.
\vskip .1cm
We observe that
$$I_1(u)=\int_{(0,\infty)\cap(0,u)}\left(u-t\right)^{2n-1}tdt=u^{2n+1}B(2n,2)$$
where $B(\cdot,\cdot)$ is the usual beta function.
\vskip .1cm
The lemma follows from the estimate of $I_1$ and the following relation.
\begin{equation}\label{kinduction}
I_{n-k+1}(u)=B(2k,2)I_{n-k}(u),\,\, k=0,\ldots,n-1.
\end{equation}
To prove (\ref{kinduction}), we recall that 
$$I_{n-k+1}(u):=\int_{A_{n,k+1}}\left(u-\sum_{j=1}^{n-k+1}t_j\right)^{2k-1}\left(\prod_{j=1}^{n-k+1}t_j\right)dt_1\ldots t_{n-k+1}.$$
Put $t_{n-k+1}=\left(u-\sum_{j=1}^{n-k}t_j\right)t$. Then we obtain
\begin{eqnarray*} I_{n-k+1}(u) &=& \int_{A_{n,k}}\left(u-\sum_{j=1}^{n-k}t_j\right)^{2k+1}\left(\prod_{j=1}^{n-k}t_j\right)\left(\int_0^1\left(1-t\right)^{2k-1}tdt\right)dt_1\ldots t_{n-k}\\ &=& B(2k,2)I_{n-k}(u).
\end{eqnarray*}
\end{proof}
Let us now prove Theorem $1.1,$ i.e. the conjecture in the case of integer exponents.
Let $s=n>1$ be an integer. Let $F\in A^2(\Pi^+)$. We recall with Proposition \ref{PW} that $$F(z)=\int_0^\infty e^{izt}f(t)dt,\,\,z\in \Pi^+$$
with $$\int_{\Pi^+}|F(z)|^2dV(z)=\pi\int_0^\infty\frac{|f(t)|^2}{t}dt.$$
We observe that $\|F\|_{A_{2n-2}^{2n}}=\|F^n\|_{A_{2n-2}^{2}}$. We can write
\begin{eqnarray*}
F^n(z) &=& \int_{(0,\infty)^n}e^{iz(t_1+t_2+\ldots+t_n)}f(t_1)\ldots f(t_n)dt_1\ldots dt_n\\ &=& \int_0^\infty e^{izu}g(u)du
\end{eqnarray*}
where $$g(u)=\int_{A_{n-1}}f(u-\sum_{j=2}^nt_j)f(t_2)\ldots f(t_n)dt_2\ldots dt_n,$$
$A_{n-1}=\{(t_2,\ldots,t_n)\in (0,\infty)^{n-1}:u-\sum_{j=2}^nt_j>0\}$.
\vskip .1cm
By Proposition \ref{PW}, we only need to estimate
$$\frac{2\pi\Gamma(2n-1)}{2^{2n-1}}\int_0^\infty\frac{|g(u)|^2}{u^{2n-1}}du.$$
Using H\"older's inequality, we easily obtain
$$|g(u)|^2\le M_n(u)\times L_n(u)$$
where $$M_n(u):=\int_{A_{n-1}}\frac{|f(u-\sum_{j=2}^nt_j)|^2}{u-\sum_{j=2}^nt_j}\times\frac{|f(t_2)|^2}{t_2}\times\ldots\times\frac{|f(t_n)|^2}{t_n}dt_2\ldots dt_n$$
and using Lemma \ref{keylemma},
$$L_n(u):=\int_{A_{n-1}}\left(u-\sum_{j=2}^nt_j\right)t_2\times\ldots\times t_ndt_2\ldots dt_n=\frac{u^{2n-1}}{\Gamma(2n)}.$$
Therefore,
\begin{eqnarray*}
\int_{\Pi^+}|F(x+iy)|^{2n}y^{2n-2}dxdy &=& \frac{2\pi\Gamma(2n-1)}{2^{2n-1}}\int_0^\infty\frac{|g(u)|^2}{u^{2n-1}}du\\ &\le& \frac{2\pi\Gamma(2n-1)}{2^{2n-1}}\times\frac{\pi^{-n}}{\Gamma(2n)}\left(\pi\int_0^\infty\frac{|f(t)|^2}{t}dt\right)^n\\ &=& \frac{\pi^{1-n}}{2^{2n-2}(2n-1)}\left(\int_{\Pi^+}|F(x+iy)|^2dxdy\right)^n.
\end{eqnarray*}
The proof of the estimate (\ref{bsconject}) in the case where $s\ge 1$ is an integer is complete.
\end{proof}
\vskip 1truemm
We next prove Proposition 1.2, which says that the conjecture holds for $s\rightarrow\infty$. We state it in the following more precise form.
\begin{prop}
For $s\ge 1$, define $$\Phi(s):=\frac {\max \limits_{F\in A^2 (\Pi^+), F\not \equiv 0} \frac {\int_{\Pi^+} |F(x+iy)|^{2s}y^{2s-2}dxdy}{\left (\int_{\Pi^+} |F(x+iy)|^2dxdy \right )^s}}{\frac {\pi^{1-s}}{(2s-1)2^{2s-2}}}.$$
Then the following hold.
\begin{itemize}
\item[(a)] For any $n\le s\le n+1$ where $n=1, 2,\cdots$, it holds that 
\begin{equation}\label{eq:onetwoineq}\frac{2s-1}{2n+1}\le \Phi(s)\le \frac{2s-1}{2n-1}.
\end{equation}
In particular, we have the following.
\item[(b)] The conjecture of Lieb-Solovej is asymptotically true, in the sense that
\begin{equation}
\lim \limits_{s\rightarrow \infty}\label{as}
\frac {\max \limits_{F\in A^2 (\Pi^+), F\not \equiv 0} \frac {\int_{\Pi^+} |F(x+iy)|^{2s}y^{2s-2}dxdy}{\left (\int_{\Pi^+} |F(x+iy)|^2dxdy \right )^s}}{\frac {\pi^{1-s}}{(2s-1)2^{2s-2}}}=1.
\end{equation}
\end{itemize}
\end{prop}
\begin{proof}
We note that (\ref{as}) follows from (\ref{eq:onetwoineq}). Hence we only prove the latter. In fact, since by Theorem 2.5, $\Phi (n)=1$ for any $n=1, 2, \cdots,$ it suffices to show that 
$$\frac{2s-1}{2n+1}\Phi(n+1)\le \Phi(s)\le \frac{2s-1}{2n-1}\Phi (n)$$
for any $n\le s\le n+1$ where $n=1, 2,\cdots.$ 
\vskip .1cm
Using the pointwise estimate in Corollary \ref{cor:pointwiseimproved}, we first obtain
$$\int_{\Pi^+} |F(x+iy)|^{2(n+1)}y^{2(n+1)-2}dxdy$$
$$= \int_{\Pi^+} |F(x+iy)|^{2s}\left(|F(x+iy)|y\right)^{2n+2-2s}y^{2s-2}dxdy$$
$$\le \frac{\pi^{s-n-1}}{2^{2n+2-2s}}\left(\int_{\Pi^+} |F(x+iy)|^{2}dxdy\right)^{n+1-s}\times \int_{\Pi^+} |F(x+iy)|^{2s}y^{2s-2}dxdy.$$
It follows that
\begin{eqnarray*}
\Phi(n+1) &\le& \frac{\pi^{s-n-1}}{2^{2n+2-2s}}\times \frac{(2n+1)2^{2n}}{\pi^{-n}}\Phi(s)\\
&=& \frac{2n+1}{2s-1}\Phi(s).
\end{eqnarray*}
Similarly, one obtains that
$$\Phi(s)\le \frac{2s-1}{2n-1}\Phi(n).$$
Hence, $$\frac{2s-1}{2n+1}\Phi(n+1)\le\Phi(s)\le \frac{2s-1}{2n-1}\Phi(n)$$
from which (\ref{eq:onetwoineq}) follows.
\end{proof}
\noindent

\section{The first test of the Lieb-Solovej conjecture. Consequences of the conjecture}
\subsection{The test on the powers of the Bergman kernel function. Proof of Theorem 1.3}
Following Proposition 2.1 and Corollary 2.5, we must test out the Lieb-Solovej conjecture on the powers of the Bergman kernel function.
\vskip 1truemm
\noindent
{\textbf {Question.}} Let $r, s >1.$ Prove or disprove the following estimate
\begin{equation}\label{questionpower}
\int_{\Pi^+} \frac {y^{2s-2}}{\left \vert x+i(y+1)\right \vert^{2rs}}dxdy\leq C_s \left (\int_{\Pi^+} \frac {dxdy}{\left \vert x+i(y+1)\right \vert^{2r}}   \right )^s.
\end{equation}
Are there values of $r>1$ for which equality holds in (\ref{questionpower})?

\vskip .1truecm
By Proposition 2.1, the estimate (\ref{questionpower}) is equivalent to the following inequality for the Gamma function
\begin{equation}\label{pow}
\frac {\Gamma (2s)\Gamma (2s(r-1))}{\left (\Gamma (rs) \right )^2}\leq \left (\frac {\Gamma (2(r-1))}{(\Gamma (r))^2} \right )^s,
\end{equation}
which is an equality when $r=2.$

It follows from Theorem 2.7 that this inequality is true when $s$ is a positive integer. We record the following corollary.
\begin{cor}
For all integers $n=2, 3, \cdots$ and real numbers $r>1,$ the following estimate holds.
$$\frac {\Gamma (2n)\Gamma (2n(r-1))}{\left (\Gamma (nr) \right )^2}\leq \left (\frac {\Gamma (2(r-1))}{(\Gamma (r))^2} \right )^n.$$
\end{cor}
The test is indeed positive according to Theorem 1.3. This result is induced by the following theorem.
\begin{thm}
The estimate $(\ref{pow})$ is valid for all $r, s>1.$ This estimate is an equality if and only if $r=2.$
\end{thm}

\begin{proof}
Without loss of generality, we assume that $r\neq 2.$ Taking the logarithm, we must prove the following estimate
\begin{equation}\label{concave}
\log \Gamma (2s)-s\log \Gamma (2) +\log \Gamma (2s(r-1))-s\log \Gamma (2(r-1))
\end{equation}
$$-2[\log \Gamma (rs)-s\log \Gamma (r)]\leq 0 
\quad \quad (r, s>1).$$
In fact, it suffices to show that for every $s>1,$ the $\mathcal C^\infty$ function $g=g_s$ defined on $(0, \infty)$ by
$$g(u)=\log \Gamma (us)-s\log \Gamma (u)$$ 
is concave. We prove that for every $s>1,$ this function $\hskip 1truemm g$ 
is strictly concave. This will imply that the inequality (\ref{concave}) (and equivalently, the inequality (\ref{pow})) is strict except for $r=2.$ We adopt the usual notation
$$\psi (x)=\left (\log \Gamma \right )'(x).$$
We have
$$g'(u)=s\psi (us)-s\psi (u)$$
and
$$g''(u)=s^2\psi' (us)-s\psi' (u)=s[s\psi' (us)-\psi' (u)].$$
To obtain that $g'' (u) <0,$ it is enough to prove that $s\psi' (us)-\psi' (u)< 0 \quad (u>0).$ This reduces to the following lemma.

\begin{lemma}
The positive function $h,$ defined on $(0, \infty)$ by $h(t):=t\psi' (t),$ is strictly decreasing.
\end{lemma}

\begin{proof}[Proof of the Lemma]
We have
$$h'(t)=\psi' (t)+t\psi'' (t).$$
To simplify the notation, we call $\eta$ this derivative function. We shall prove that $\eta(x)<0$ for every $x>0.$ We recall the asymptotic expansions as $x\rightarrow \infty$:
$$\psi' (x)\sim \frac 1x+\frac 1{2x^2}+\mathcal O (\frac 1{x^3)}$$
and 
$$\psi'' (x)\sim -\frac 1{x^2}-\frac 1{x^3}+\mathcal O (\frac 1{x^4)}.$$
This implies that $\eta(x)=-\frac 1{2x^2}+\mathcal O (\frac 1{x^3}).$ So $\eta(\infty)=\lim \limits_{x\rightarrow \infty} \eta(x)=0.$ We also recall the following formulas (cf. e.g. \cite{AS}):
$$\psi'(x+1)=\psi' (x)-\frac 1{x^2}$$
and 
$$\psi'' (x+1)=\psi'' (x)+\frac 2{x^2}.$$
This yields
$$\eta(x+1)-\eta(x)=\psi'' (x)=\psi''(x)+\frac 1{x^2}+\frac 2{x^3}.$$
From the formula (cf. e.g. \cite{AS})
$$\psi (x)+\gamma=\int_0^\infty \frac {e^{-t}-e^{-xt}}{1-e^{-t}}dt,$$
where $\gamma$ denotes the Euler constant, we deduce that
$$\psi' (t)=\int_0^\infty \frac {te^{-xt}}{1-e^{-t}}dt$$ 
and 
$$\psi'' (t)=-\int_0^\infty \frac {t^2e^{-xt}}{1-e^{-t}}dt.$$
We obtain
$$
\begin{array}{clcr}
\eta(x+1)-\eta(x)&=-\int_0^\infty \frac {t^2e^{-xt}}{1-e^{-t}}dt+\int_0^\infty te^{-xt}dt+\int_0^\infty t^2e^{-xt}dt\\
&=\int_0^\infty \left (-\frac {te^{t}}{e^t-1}+1+t    \right )te^{-xt}dt\\
&=\int_0^\infty \left (1-\frac t{e^t-1} \right )te^{-xt}dt>0,
\end{array}
$$
since $1-\frac t{e^t-1}>0$ for every positive $t.$ Next we get $g(x+1)>g(x)$ for every $x>0.$ This gives 
$$\eta(x)< \eta(x+n) \quad \quad (n=1, 2, \cdots).$$
Letting $n$ tend to $\infty,$ we conclude that $h'(x)=\eta(x)< \eta(\infty)=0$ for every $x>0.$ The proof is complete.
\end{proof}
The proof of the inequality (3.2) is then complete. This finishes the proof of Theorem 3.2.
\end{proof}

\subsection{Some consequences of the Lieb-Solovej conjecture}
In this subsection, we assume that the Lieb-Solovej conjecture is true and we shall first draw as a consequence 
the following dual estimate.
\begin{cor}
Assume that the Lieb-Sobolev conjecture is true. Let $s>1.$ Then
the identity operator is bounded from $A^{\frac {2s}{2s-1}}_{-\frac {2(s-1)}{2s-1}} (\Pi^+)$ to $A^2 (\Pi^+)$ with operator norm $C_s^{\frac 1{2s}}.$
\end{cor}
\begin{proof}
For all $F\in A^2 (\Pi^+)$ and $G\in A^{\frac {2s}{2s-1}}_{-\frac {2(s-1)}{2s-1}} (\Pi^+),$ an application of the H\"older inequality gives
$$\left |\int_{\Pi^+} F(x+iy)\overline {G(x+iy)}dxdy\right |$$
$$=\left |\int_{\Pi^+} F(x+iy)y^{\frac {2s-2}{2s}}\overline {G(x+iy)}y^{-\frac {2s-2}{2s}}dxdy\right |$$
$$\leq \left (\int_{\Pi^+} |F(x+iy)|^{2s}y^{2s-2}dxdy\right )^{\frac 1{2s}}\left (\int_{\Pi^+} |G(x+iy)|^{\frac {2s}{2s-1}}y^{-\frac {2s-2}{2s-1}}dxdy\right )^{\frac {2s-1}{2s}}.$$
We deduce from the conjecture that
$$\sup \limits_{F\in A^2 (\Pi^+), \hskip 1truemm \left \Vert F\right \Vert_{A^2 (\Pi^+)}=1}\left |\int_{\Pi^+} F(x+iy)\overline {G(x+iy)}dxdy\right |$$
$$\leq C_s^{\frac 1{2s}}\left (\int_{\Pi^+} |G(x+iy)|^{\frac {2s}{2s-1}}y^{-\frac {2s-2}{2s-1}}dxdy\right )^{\frac {2s-1}{2s}}.$$ 
We recall that the dual of the (Hilbert-)Bergman space $A^2 (\Pi^+)$ is $A^2 (\Pi^+)$ with respect to the duality pairing
$$<F, G>=\int_{\Pi^+} F(x+iy)\overline {G(x+iy)}dxdy.$$ 
We conclude that for every $G\in A^{\frac {2s}{2s-1}}_{-\frac {2(s-1)}{2s-1}} (\Pi^+),$ we have
$$\left \Vert G\right \Vert_{A^2 (\Pi^+)}=\left (\int_{A^2 (\Pi^+)} |G(x+iy)|^2dxdy \right )^{\frac 12}$$
$$\leq C_s^{\frac 1{2s}}\left (\int_{\Pi^+} |G(x+iy)|^{\frac {2s}{2s-1}}y^{-\frac {2s-2}{2s-1}}dxdy\right )^{\frac {2s-1}{2s}}.$$
So the identity operator is bounded from $A^{\frac {2s}{2s-1}}_{-\frac {2(s-1)}{2s-1}} (\Pi^+)$ to $A^2 (\Pi^+)$ with operator norm $C_s^{\frac 1{2s}}.$ 
\end{proof}
\vskip 1truemm
We next prove Theorem 1.4. Let $F\in A^2 (\Pi^+)$ such that $\left \Vert F\right \Vert_{A^2 (\Pi^+)}=1.$ Assume that $(1)$ is true, i.e. 
 equivalently
$$\left (\int_{\Pi^+} \left (\left \vert F(x+iy)\right \vert y\right )^{2s-2}\left \vert F(x+iy)\right \vert^2dxdy\right )^{\frac 1{2s-2}} \leq \frac {1}{2\sqrt \pi (2s-1)^{\frac 1{2s-2}}}.$$
The measure $\left \vert F(x+iy)\right \vert^2dxdy$ is a probability measure on $\Pi^+.$ Letting $s$ tend to $1,$ we obtain (cf. e.g. \cite{R}, page 71, exercise 5):
$$\exp \left (\int_{\Pi^+} \log \left [\left \vert F(x+iy)\right \vert y  \right ]\left \vert F(x+iy)\right \vert^2dxdy \right )\leq \frac 1{2e\sqrt \pi}.$$
Taking the logarithm of both sides, we obtain the estimate (\ref{step}). In other words, $(2)$ is true.

Conversely, assume that $(2)$ is true. Taking the logarithm of both sides of (\ref{local}), we obtain the following equivalent form
$$\varphi (s):=\log \left (\int_{\Pi^+}|F(x+iy)|^{2s}y^{2s-2}dxdy \right )-(1-s)\log \pi +\log (2s-1)$$
$$+(2s-2)\log 2\leq 0.$$
Since $\varphi (1)=0,$ it suffices to prove that $\varphi' (1)\leq 0.$ For $s\geq 1,$ we have
$$
\begin{array}{clcr}
\varphi' (s)&=\frac {\frac d{ds}\int_{\Pi^+}|F(x+iy)|^{2s}y^{2s-2}dxdy}{\int_{\Pi^+}|F(x+iy)|^{2s}y^{2s-2}dxdy}+\log \pi+\frac 2{2s-1}+2\log 2\\
&=\frac {\frac d{ds}\int_{\Pi^+}e^{2(s-2)\left (\log (|F(x+iy)|y)\right )}|F(x+iy)|^2dxdy}{\int_{\Pi^+}|F(x+iy)|^{2s}y^{2s-2}dxdy}+\log \pi+\frac 2{2s-1}+2\log 2\\
&=\frac {\int_{\Pi^+}2\log (|F(x+iy)|y)|F(x+iy)|^{2s}y^{2s-2}dxdy}{\int_{\Pi^+}|F(x+iy)|^{2s}y^{2s-2}dxdy}+\log \pi+\frac 2{2s-1}+2\log 2.
\end{array}$$
In particular,
$$\varphi' (1)=2\int_{\Pi^+}\log (|F(x+iy)|y)|F(x+iy)|^{2}dxdy+\log \pi+1+2\log 2.$$
The required estimate $\varphi' (1)\leq 0$ is equivalent to
$$\int_{\Pi^+}\log (2\sqrt \pi|F(x+iy)|y)|F(x+iy)|^{2}dxdy\leq -1,$$
which reduces to (2). In other words, $(1)$ is true.
\vskip 5truemm
We finally prove Theorem 1.5,  which provides   a less precise estimate than the estimate (\ref{step}).
The measure $\left \vert F(x+iy)\right \vert^2dxdy$ is a probability measure on $\Pi^+.$ The function $\varphi (t):=\log \frac 1{2\sqrt \pi t}$ is a convex function on $(0, \infty).$ An application of Jensen's inequality (cf. e.g. \cite{R}, Theorem 3.3) gives
$$\int_{\Pi^+} \log \left [\frac 1{2\sqrt \pi|F(x+iy)|y}  \right ]|F(x+iy)|^{2}dxdy$$
$$\geq \log \frac 1{2\sqrt \pi\int_{\Pi^+} \left (|F(x+iy)|y\right )|F(x+iy)|^{2}dxdy}.$$
Now, by the Schwarz inequality, we obtain
$$\int_{\Pi^+} \left (|F(x+iy)|y\right )|F(x+iy)|^{2}dxdy$$
$$ \leq \left (\int_{\Pi^+} \left (|F(x+iy)|y\right )^2|F(x+iy)|^{2}dxdy \right )^{\frac 12}\leq \left (\frac {\pi^{-1}}{3\times 2^2} \right )^{\frac 12}=\frac 1{2\sqrt {3\pi}},$$
since the conjecture is true for $s=2.$ We conclude that
$$\int_{\Pi^+} \log \left [\frac 1{2\sqrt \pi|F(x+iy)|y}  \right ]|F(x+iy)|^{2}dxdy \geq \log \sqrt 3=\frac {\log 3}2.$$

\section{The Lieb-Solovej conjecture on the unit disc}
\subsection{The statement of the conjecture}
We denote by $\mathbb D$ the unit disc in the complex plane and by $dm$ the Lebesgue area measure in the complex plane. Via a transfer principle from the upper half-plane $\Pi^+$ to the unit disc $\mathbb D,$ the Lieb-Solovej conjecture takes the following form on $\mathbb D.$
\vskip 2truemm
\noindent
{\textbf {Conjecture.}} Let $s>1.$ For every $G\in A^2 (\mathbb D),$ we have
$$\int_{\mathbb D} \left \vert G(z)\right \vert^{2s}(1-|z|^2)^{2s-2}dm(z)\leq \frac {\pi^{1-s}}{2s-1}\left (\int_{\mathbb D}\left \vert G(z)\right \vert^{2}dm(z) \right )^s,$$
with equality when $G(z)$ is a Bergman kernel function $\frac 1\pi \left (1-z\cdot \overline {z_0}\right )^{-2},$ \hskip 1truemm
$z_0\in \mathbb D,$ of the unit disc $\mathbb D.$ 
\vskip 2truemm
\begin{proof}
We recall the Lieb-Solovej conjecture on the upper half-plane $\Pi^+.$
\begin{equation}\label{conj}
\int_{\Pi^+} \left \vert F(x+iy)\right \vert^{2s}y^{2s-2}dxdy
\end{equation}
$$\leq \frac {\pi^{1-s}}{(2s-1)2^{2s-2}}\left (\left \vert F(x+iy)\right \vert^{2}dxdy \right )^s \quad \quad (F\in A^2 (\Pi^+)).$$
We apply the linear fractional transformation $\Phi$ from $\mathbb D$ to $\Pi^+:$
$$z=x+iy=\Phi (w):=i\frac {1+w}{1-w}.$$
Then $y=\Im m \hskip 1truemm z=\Im m \left (i\frac {1+w}{1-w}\right )=\frac {1-|w|^2}{|1-w|^2}$ and $\Phi' (w)=\frac {2i}{(1-w)^2}.$ So the estimate (\ref{conj}) takes the form
$$\int_{\mathbb D} \left \vert F(\Phi(w))\right \vert^{2s}\left (\frac {1-|w|^2}{|1-w|^2}\right )^{2s-2}\frac 4{|1-w|^4}dm(w)$$
$$=2^{-2s+2}\int_{\mathbb D} \left \vert F(\Phi(w))\right \vert^{2s}\left (\frac {2}{|1-w|^2}\right )^{2s}(1-|w|^2)^{2s-2}dm(w)$$
$$=2^{-2s+2}\int_{\mathbb D} \left \vert F(\Phi(w))\Phi' (w)\right \vert^{2s}(1-|w|^2)^{2s-2}dm(w)$$
$$\stackrel {(\ref{conj})} \leq \frac {\pi^{1-s}}{(2s-1)2^{2s-2}}\int_{\mathbb D} \left \vert F(\Phi(w))\Phi' (w)\right \vert^{2}dm(w).$$
Without loss of generality, we take $G=(F\circ \Phi)\Phi':$ the result follows.\\
For equality, apply the change of variable formula in the Bergman kernel.
\end{proof}

As mentioned in the introduction, Bayart, Brevig, Haimi, Ortega-Cerd\`a and Perfekt \cite{BBHOP} proved differently the Lieb-Solovej conjecture for $s=2, 3, 4, \cdots$ and they even settled it for $s=\frac 32, \frac 52, \frac 72, \cdots.$ More precisely, let $\alpha >1$ and $0<p<\infty,$ and define the Bergman $B^p_\alpha (\mathbb D)$ as the space of holomorphic $f$ on the unit disc $\mathbb D$ whose norm 
$$\left \Vert f\right \Vert_{B^p_\alpha (\mathbb D)}:=\left (\int_{\mathbb D} |f(w)|^p (\alpha-1)\left (1-|w|^2 \right )^{\alpha-2}\frac {dm(z)}\pi \right )^{\frac 1p}$$
is finite. For $\alpha_0=\frac {1+\sqrt {17}}4,$ these authors prove the following theorem.

\begin{thm}\cite[Theorem 1]{BBHOP}
Let $\alpha\geq \alpha_0$ and $0<p<\infty.$ For every $f\in B^p_\alpha (\mathbb D),$
\begin{equation}\label{bay}
\left \Vert f\right \Vert_{B^{\frac {p(\alpha+1)}\alpha}_{\alpha+1} (\mathbb D)}\leq \left \Vert f\right \Vert_{B^p_\alpha (\mathbb D)}.
\end{equation}
Moreover, if $\alpha>\alpha_0,$ equality holds in {\rm {(\ref{bay})}} if and only if there exist a complex constant $C$ and a point $\xi$ in $\mathbb D$ such that $f(w)=\frac C{(1-\bar \xi\cdot w)^{\frac {2\alpha}p}}.$ 
\end{thm}

From this theorem, they deduce the following corollary, which settles the Lieb-Solovej conjecture for $s=2, 3, 4,\cdots$ and $s=\frac 32, \frac 52, \frac 72,\cdots.$
\begin{cor}\cite[Corollary 2]{BBHOP}\label{cor:bayart}
Let $f\in B^2_2 (\mathbb D).$ Then
$$\left \Vert f\right \Vert_{B^2_2 (\mathbb D)}\geq \left \Vert f\right \Vert_{B^3_3 (\mathbb D)} \geq \left \Vert f\right \Vert_{B^4_4 (\mathbb D)} \geq \cdots.$$
\end{cor}

For the proof of Theorem 4.1, Bayart et al. solve a minimization problem in a Sobolev space $W^{1, 2} (\mathbb D).$

\subsection{A second test of the conjecture}
We must test out the Lieb-Solovej conjecture for the unit disc on the analytical monomials $z^n \quad (n=1, 2,\cdots).$
\vskip 2truemm
\noindent
{\textbf {Question.}} Let $s>1$ and $n=1, 2,\cdots$ Does the following estimate hold?
\begin{equation}\label{monomial}
\int_{\mathbb D} \left \vert z^n\right \vert^{2s}(1-|z|^2)^{2s-2}dm(z)\leq \frac {\pi^{1-s}}{2s-1}\left (\int_{\mathbb D}\left \vert z^n\right \vert^{2}dm(z) \right )^s.
\end{equation}
\vskip 2truemm
We move to the polar coordinates. The right hand side of (\ref{monomial}) is equal to
$$\frac {\pi^{1-s}}{2s-1}\left (2\pi \int_0^1 r^{2n+1}dr \right )^s=\frac {\pi^{1-s}}{2s-1}\left (\frac \pi{n+1} \right )^s=\frac \pi{(2s-1)(n+1)^s}.$$
The left hand side is equal to
$$L:=2\pi\int_0^1 r^{2ns}(1-r^2)^{2s-2}rdr.$$
We apply the change of variable $r^2=\rho, \hskip 2truemm 2rdr=d\rho.$ This gives
$$L=\pi\int_0^1 \rho^{ns}(1-\rho)^{2s-2}d\rho=\pi B(ns+1, 2s-1)=\pi \frac {\Gamma (ns+1)\Gamma (2s-1)}{\Gamma ((n+2)s)}.$$
Here, $B(\cdot, \cdot)$ denotes the Euler Beta function.

So the estimate (\ref{monomial}) is equivalent to the following inequality for the Gamma function:
\begin{equation}\label{monogam}
\frac {\Gamma (ns+1)\Gamma (2s)}{\Gamma ((n+2)s)}\leq \frac 1{(n+1)^s}\quad \quad (s>1, \hskip 1truemm n=1, 2,\cdots).
\end{equation}

It appears that the inequality (\ref{monogam}) (and hence the estimate (\ref{monomial})) is true.

\begin{thm}
The estimate $(\ref{monogam})$ is valid and strict for all $s>1$ and $n=1, 2,\cdots.$
\end{thm}

\begin{proof}
\noindent
Taking the logarithm of both sides of the inequality (\ref{monogam}), we wish to prove that given $n=1, 2,\cdots,$ we have
$$\varphi (s):=\log n +\log s+\log \Gamma (ns)+\log \Gamma (2s)-\log \Gamma ((n+2)s)+s\log (n+1)
<   0 $$
for every $s>1.$ Since $\varphi (1)=0,$ it suffices to show that the derivative $\varphi'(s)$ of the function $\varphi (s)$ is negative. We have
$$\varphi' (s)=\frac 1s+\frac {n\Gamma' (ns)}{\Gamma (ns)}+\frac {2\Gamma' (2s)}{\Gamma (2s)}-\frac {(n+2)\Gamma' ((n+2)s)}{\Gamma ((n+2)s)}+\log (n+1).$$
We recall the following well-known identity (cf. e.g. \cite{AS}):
$$\frac {\Gamma' (x)}{\Gamma (x)}=-\frac 1x +\log x -\int_0^\infty \left (\frac 1{e^t-1}-\frac 1t \right )e^{-xt}dt \quad \quad (x>0).$$
So
$$
\begin{array}{clcr}
\varphi'(s)&=\frac 1s+n\left \{-\frac 1{ns} +\log (ns) -\int_0^\infty \left (\frac 1{e^t-1}-\frac 1t \right )e^{-nst}dt     \right \}\\
&+2\left \{-\frac 1{2s} +\log (2s) -\int_0^\infty \left (\frac 1{e^t-1}-\frac 1t \right )e^{-2st}dt     \right \}\\
&-(n+2)\left \{-\frac 1{(n+2)s} +\log ((n+2)s) -\int_0^\infty \left (\frac 1{e^t-1}-\frac 1t \right )e^{-(n+2)st}dt     \right \}\\
&+\log (n+1)\\
&=n\log (ns)+2\log (2s) -(n+2)\log ((n+2)s)+\log (n+1)\\
&-\int_0^\infty \left (\frac 1{e^t-1}-\frac 1t \right )\left \{ne^{-nst}+2e^{-2st}-(n+2)e^{-(n+2)st}\right \}dt\\
&=n\log n+2\log 2 -(n+2)\log (n+2)+\log (n+1)\\
&+\frac 1s\int_0^\infty \left (\frac 1{e^t-1}-\frac 1t \right )\left \{\frac d{dt} \left [e^{-nst}+e^{-2st}-e^{-(n+2)st}\right ]\right \}dt.
\end{array}
$$
An integration by parts gives
$$\varphi'(s)=n\log n+2\log 2 -(n+2)\log (n+2)+\log (n+1)$$
$$+\frac 1{2s}-\frac 1s \int_0^\infty \left \{\frac d{dt} \left (\frac 1{e^t-1}-\frac 1t \right )\right \}\left [e^{-nst}+e^{-2st}-e^{-(n+2)st}\right ]dt.$$
We next show that $\varphi' (s)<  0 \quad (s>1 \hskip 1truemm {\rm {and}} \hskip 1truemm n=1, 2,\cdots).$ It is easy to check that $e^{-nst}+e^{-2st}-e^{-(n+2)st}>0$ and $\frac d{dt} \left (\frac 1{e^t-1}-\frac 1t \right )\geq 0.$ It suffices to prove that
\begin{equation}\label{n}
n\log n+2\log 2 -(n+2)\log (n+2)+\log (n+1)+\frac 12< 0 \quad \quad (n=1, 2,\cdots).
\end{equation}
To this aim, we consider the function
$$h (y):=y\log y +2\log 2-(y+2)\log (y+2)+\log (y+1)+\frac 12 \quad \quad (y\geq 1).$$
Its first order derivative is equal to
$$h' (y)=\log y-\log (y+2)+\frac 1{y+1}$$
and its second order derivative is equal to
$$h'' (y)=\frac 1y-\frac 1{y+2}-\frac 1{(y+1)^2}=\frac {y^2+2y+2}{y(y+2)(y+1)^2}.$$
Clearly, $h'' (y)>0$ for every $y\geq 1,$ so that $h'(y)$ increases from $h'(1)=-\log 3 +\frac 12$ to $\lim \limits_{y\rightarrow \infty} h'(y)=0.$ In particular, $h'(y)<  0$ for every $y\geq 1.$ We conclude that $h(y)<  h(1)=3\log 2-3\log 3+\frac 12< 0.$ The estimate (\ref{n}) follows. The proof of the theorem is complete.
\end{proof}

\subsection{Consequences of the conjecture for $s$ close to 1 on the unit disc.}
Via the same transfer principle, the conjecture (\ref{step}) for $s$ close to 1 on the upper half-plane $\Pi^+$ takes the following form on the unit disc $\mathbb D:$
\begin{equation}\label{step1}
\int_{\mathbb D} \log \left [\frac 1{\sqrt \pi|G(z)|(1-|z|^2)}  \right ]|G(z)|^{2}dm (z)\geq 1
\end{equation}
for every $G\in A^2 (\mathbb D)$ such that $\left \Vert F\right \Vert_{A^2 (\mathbb D)}=1.$

Theorem 3.6 implies the following result on the unit disc.

\begin{cor}
For every $G\in A^2 (\mathbb D)$ such that $\left \Vert G\right \Vert_{A^2 (\mathbb D)}=1,$ we have
$$\int_{\mathbb D} \log \left [\frac 1{\sqrt \pi|G(z)|(1-|z|^2)}  \right ]|G(z)|^{2}dm (z)\geq \frac {\log 3}2.$$
\end{cor}

\section{The obtained bounds}
\subsection{A preliminary bound}
The following corollary is a direct consequence of Corollary 2.3.
\begin{cor}
For every $F\in A^2 (\Pi^+),$ we have
$$\int_{\Pi^+} |F(x+iy)|^{2s}y^{2s-2}dxdy\leq \frac {\pi^{1-s}}{2^{2s-2}}\left (\int_{\Pi^+} |F(x+iy)|^2dxdy \right )^s.$$
\end{cor}

\subsection{An unsuccessful attempty via Minkowski's integral inequality}
A less successful attempt is via Minkowski's integral inequality. We obtain the following result.
\begin{prop}
Given $s>1$ and $F\in A^2 (\Pi^+),$ for every $\nu >1,$ we have
$$\int_{\Pi^+} |F(x+iy)|^{2s}y^{2s-2}dxdy$$
$$\leq C_s\frac {\nu^s \Gamma (2s)\Gamma ((\nu-1)s)}{\left (\Gamma \left (\frac {(\nu+1)s}2 \right )\right )^2}\left (\int_{\Pi^+} |F(x+iy)|^2dxdy \right )^s.$$ 
In particular, for $\nu=3,$ we obtain
$$\int_{\Pi^+} |F(x+iy)|^{2s}y^{2s-2}dxdy\leq 3^s C_s\left (\int_{\Pi^+} |F(x+iy)|^2dxdy \right )^s.$$ 
\end{prop}

\begin{proof}
For every $\nu >1,$ we again rely on the reproducing property of the weighted Bergman kernel, we have
$$\left [F(x+iy)\right ]^2=\frac {2^{\nu-1}\nu}{\pi}\int_{\Pi^+} \frac {v^{\nu-1}}{(x-u+i(y+v))^{\nu+1}}\left [F(u+iv)\right ]^2dudv.$$
When we apply Minkowski's integral inequality and Proposition 2.1, we obtain
$$\left (\int_{\Pi^+} |F(x+iy)|^{2s}y^{2s-2}dxdy\right )^{\frac 1s}$$
$$\leq \frac {2^{\nu-1}\nu}{\pi}\left (\int_{\Pi^+}\left (\int_{\Pi^+} \frac {v^{\nu-1}}{|x-u+i(y+v)|^{\nu+1}}\left \vert F(u+iv)\right \vert^2dudv\right )^s y^{2s-2}dxdy\right )^{\frac 1s}$$
$$\leq \frac {2^{\nu-1}\nu}{\pi}\int_{\Pi^+}\left (\int_{\Pi^+} \left [\frac {v^{\nu-1}}{|x-u+i(y+v)|^{\nu+1}}\left \vert F(u+iv)\right ]^2 \right ]^s y^{2s-2}dxdy\right )^{\frac 1s}dudv$$
$$= \frac {2^{\nu-1}\nu}{\pi}\int_{\Pi^+}\left (\int_{\Pi^+} \frac {y^{2s-2}}{|x-u+i(y+v)|^{(\nu+1)s}} dxdy\right )^{\frac 1s}\left \vert F(u+iv)\right \vert^2 v^{\nu-1}dudv$$
$$= \frac {2^{\nu-1}\nu}{\pi}\int_{\Pi^+}\left (\frac {4\pi\Gamma (2s-1)\Gamma ((\nu-1)s)}{2^{(\nu+1)s}\left (\Gamma \left (\frac {(\nu+1)s}2 \right )  \right )^2v^{(\nu-1)s}} \right )^{\frac 1s}\left \vert F(u+iv)\right \vert^2 v^{\nu-1}dudv$$
$$= \frac {\nu}{4\pi}\left (\frac {4\pi\Gamma (2s-1)\Gamma ((\nu-1)s)}{\left (\Gamma \left (\frac {(\nu+1)s}2 \right )  \right )^2} \right )^{\frac 1s}\int_{\Pi^+}\left \vert F(u+iv)\right \vert^2 dudv.$$
The result follows easily.
\end{proof}

\begin{rem}
Comparing the bounds of Corollary 5.1 and Proposition 5.2, namely $2s-1$ and $\inf \limits_{\nu >1} \frac {\nu^s \Gamma (2s)\Gamma ((\nu-1)s)}{\left (\Gamma \left (\frac {(\nu+1)s}2 \right )\right )^2}.$ It looks surprising that the first bound may be smaller than or equal to the second at least for $1<s<2.$ This is implied by the following proposition.
\end{rem}

\begin{prop}
For all $1<s<2,$ the following inequality holds
$$2s-1 \leq \inf \limits_{\nu >1} \frac {\nu^s \Gamma (2s)\Gamma ((\nu-1)s)}{\left (\Gamma \left (\frac {(\nu+1)s}2 \right )\right )^2}.$$
\end{prop}

\begin{proof}
In view of the convexity of the function $\log \Gamma,$ we have 
$$\frac {\Gamma (2s)\Gamma ((\nu-1)s)}{\left (\Gamma \left (\frac {(\nu+1)s}2 \right )\right )^2}\geq 1\quad \quad (s>1).$$
 Moreover, the inequality $2s-1\leq \nu^s$ holds if and only if $\nu \geq (2s-1)^{\frac 1s}.$ In this case, we have
$$2s-1\leq  \frac {\nu^s \Gamma (2s)\Gamma ((\nu-1)s)}{\left (\Gamma \left (\frac {(\nu+1)s}2 \right )\right )^2}.$$
Let us now suppose that $1<\nu\leq (2s-1)^{\frac 1s}.$ We recall the following identity (cf. e.g. \cite{AS, GR}:
\begin{equation*}
\Gamma (x)=e^{-\gamma x}\frac 1x \prod_{k=1}^\infty \frac {e^{\frac xk}}{1+\frac xk}.
\end{equation*}
This implies that
\begin{equation*}
\frac {\Gamma (2s)\Gamma ((\nu-1)s)}{\left (\Gamma \left (\frac {(\nu+1)s}2 \right )\right )^2}=\frac 2{\nu^2-1}\prod_{k=1}^\infty \frac {\left (1+\frac {(\nu+1)s}{2k} \right )^2}{\left (1+\frac {2s}k \right )\left (1+\frac {(\nu-1)s}k \right )}.
\end{equation*}
We check easily that for all $\nu, s>1$ and $k=1, 2, \cdots,$ we have
$$\frac {\left (1+\frac {(\nu+1)s}{2k} \right )^2}{\left (1+\frac {2s}k \right )\left (1+\frac {(\nu-1)s}k \right )}\geq 1$$ and hence
\begin{equation*}
\prod_{k=1}^\infty \frac {\left (1+\frac {(\nu+1)s}{2k} \right )^2}{\left (1+\frac {2s}k \right )\left (1+\frac {(\nu-1)s}k \right )}\geq 1.
\end{equation*}
It then suffices to show the inequality $2s-1\leq \frac {2\nu^s}{\nu^2-1}$ for all $1<\nu\leq (2s-1)^{\frac 1s}.$ We study the function
$$\varphi (\nu)=(2s-1)(\nu^2-1)-2\nu^s.$$
Then $\varphi' (\nu)=2(2s-1)\nu-2s\nu^{s-1}$ and $\varphi'' (\nu)=2(2s-1)-2s(s-1)\nu^{s-2}.$ The following equivalence holds
$$\varphi''(\nu)=0 \Leftrightarrow \nu^{s-2}=\frac {2s-1}{s(2s-1)}.$$
\noindent
We show easily that $\frac {2s-1}{s(2s-1)}>1$ if and only if $1<s<\frac {3+\sqrt 5}2.$\\
Assuming that $1<s<2,$ since $\nu =\left (\frac {2s-1}{s(s-1)} \right )^{\frac 1{s-2}}<1$ and $\lim \limits_{\nu \rightarrow \infty} \varphi''(\nu)=2(2s-1)>0,$ we obtain that $\varphi''(\nu)>0$ for all $\nu>1.$ So $\varphi'(s)$ increases from the positive value $2s-2$ on the interval $[1, \infty);$ in particular, $\varphi'(s)>0$ for all $1<\nu<\infty.$ Finally, we conclude that the function $\varphi (\nu)$ increases from the value $-2$ to $\infty$ on the interval $[1, \infty).$ We have
$$\varphi \left ((2s-1)^{\frac 1s}\right )=(2s-1)^{\frac 2s+1}-3(2s-1)<0$$
provided that $1<s<2$ (study the function $\psi (s)=\log (2s-1)-\frac s2\log 3).$ This proves that for all $1<\nu<(2s-1)^{\frac 1s},$ we have $\varphi (\nu)<0$ or equivalently $2s-1\leq \frac {2\nu^s}{\nu^2-1}.$ This completes the proof of the proposition.
\end{proof}

\subsection{An improved bound via the complex interpolation method}
A classical example of  interpolation via the complex method concerns $L^p$ spaces with a change of measures. We state it in our setting of the upper half-plane $\Pi^+.$

\begin{thm}\label{change} \cite{BL, SW1}
Let $1\leq p_0,\hskip 1truemm p_1 \leq \infty.$ Given two positive measurable functions (weights) $\omega_0, \hskip 2truemm \omega_1$ on $(0, \infty),$ then for every $\theta \in (0, 1),$ we have
$$[L^{p_0} (\Pi^+, \hskip 1truemm \omega_0 (y)dxdy),  L^{p_1} (\Pi^+ , \hskip 1truemm \omega_1 (y)dxdy)]_\theta=L^p (\Pi^+, \hskip 1truemm \omega (y) dxdy)$$
with equal norms, provided that
$$\frac{1}{p}=\frac{1-\theta}{p_0}+\frac{\theta}{p_1}$$
$$\omega^\frac{1}{p}=\omega_0^{\frac{1-\theta}{p_0}}\omega_1^{\frac{\theta}{p_1}}.$$
\end{thm}

We obtain the following theorem.
\begin{thm}
We suppose that $s>1$ is not an integer. Let $n$ be the integer such that $n<s<n+\frac 12$ (resp. $n+\frac 12 <s<n+1$). Then for every $F\in A^2 (\Pi^+),$ we have the estimate
$$\int_{\Pi^+} |F(x+iy)|^{2s}y^{2s-2}dxdy$$
$$\leq \frac {\pi^{1-s}}{(2n-1)^{-2s+2n+1}(2n)^{2(s-n)}2^{2s-2}}\left (\int_{\Pi^+} |F(x+iy)|^2dxdy \right )^s$$
\rm {(}resp.
$$\int_{\Pi^+} |F(x+iy)|^{2s}y^{2s-2}dxdy$$
$$\leq \frac {\pi^{1-s}}{(2n-1)^{-2s+2n+1}(2n)^{2(s-n)}2^{2s-1}}\left (\int_{\Pi^+} |F(x+iy)|^2dxdy \right )^s).$$
\end{thm}
\begin{proof}
We provide the proof for $n<s<n+\frac 12.$ In Theorem 5.5, we take $p_0=2n, \hskip 1truemm p_1=2n+1, \hskip 1truemm p=2s, \hskip 1truemm \omega_0 (y)=y^{2n-2}, \hskip 1truemm \omega_1 (y)=y^{2n-1}.$ Let $\theta\in (0, 1)$ be defined by
\begin{equation}\label{thetabis}
\frac 1s=\frac {1-\theta}n+\frac \theta{n+\frac 12}.
\end{equation}
We check easily that in the notations of Theorem 5.5, we have $\omega (y)=y^{2s-2}.$
We can then apply Theorem 5.5. First, we have trivially $\left [A^2 (\Pi^+), A^2 (\Pi^+)\right ]_\theta=A^2 (\Pi^+).$ Next, if we consider the identity operator $i(F)=F,$ then by Corollary \ref{cor:bayart}, $i$ is bounded from $A^2 \left (\Pi^+\right )$ to $L^{2n} \left (\Pi^+, y^{2n-2}dxdy\right )$ with operator norm $\left (\frac {\pi^{1-n}}{(2n-1)2^{2n-2}}\right )^{\frac 1{2n}}$ and is bounded from $A^2 \left (\Pi^+\right )$ to $L^{2n+1} \left (\Pi^+, y^{2n-1}dxdy\right )$ with operator norm $\left (\frac {\pi^{\frac 12-n}}{(2n)2^{2n-1}}\right )^{\frac 1{2n+1}},$ we obtain that
\begin{equation}\label{improved}
\left (\int_{\Pi^+} |F(x+iy)|^{2s}y^{2s-2}dxdy\right )^{\frac 1{2s}}
\end{equation}
$$\leq \left (\left (\frac {\pi^{1-n}}{(2n-1)2^{2n-2}}\right )^{\frac 1{2n}}\right )^{1-\theta}\left (\left (\frac {\pi^{\frac 12-n}}{(2n)2^{2n-1}}\right )^{\frac 1{2n+1}}\right )^{\theta}\left (\int_{\Pi^+} |F(x+iy)|^2dxdy \right )^{\frac 12}.$$
We deduce from (\ref{thetabis}) that $\theta=(1-\frac ns)(2n+1)$ and $1-\theta=-2n+\frac ns(2n+1).$ Replacing in (\ref{improved}) gives the announced estimate.
\end{proof}

\begin{rem}
\begin{enumerate} 
\item
The estimate in Theorem 5.6 clearly improves the one in Corollary 5.1.
\item
Corollary 5.6 also implies that the conjecture of Lieb-Solovej is asymptotically true, as proved in Proposition 2.7.
\end{enumerate}
\end{rem}

We next show that the bound in Theorem 5.6 is greater than the bound in the Lieb-Solovej conjecture when $s$ is not an integer. It also makes more precise the constant $\Phi (s)$ in Proposition 2.7. We recall that the Lieb-Solovej conjecture says that $\Phi (s)=1.$ In fact, Theorem 5.6 provides a more accurate upper bound of $\Phi (s)$ which lies in $\left (1, \frac {2s-1}{2n-1} \right )$ when $s\in (n, n+1);$ explicitly,
$$\Phi (s)\leq \frac {2s-1}{(2n-1)^{-2s+2n+1}(2n)^{2(s-n)}}.$$
These two assertions are consequences of the following elementary lemma.

\begin{lemma}
For $s\in (n, n+1),$ the following double inequality holds
$$\frac 1{2s-1}<\frac 1{(2n-1)^{-2s+2n+1}(2n)^{2(s-n)}}<\frac 1{2n-1}.$$
\end{lemma}

\begin{proof}
The second inequality is obvious. The second inequality follows easily from the strict concavity of the function $t\mapsto \log (2t-1).$
\end{proof}

\subsection*{Acknowledgments}  The authors express their sincere thanks to Chao-Ping Chen for showing them his proof of Lemma 3.3 and to Joaquim Ortega-Cerd\`a for indicating us the references \cite{BBHOP}, \cite{BOSZ} and \cite{B}. Corollary 2 of \cite{BBHOP} led to an improvement of Theorem 5.5 and the related bound.

\end{document}